\documentclass{amsart}
\usepackage{amssymb, amsmath, latexsym}
\usepackage{graphicx}
\usepackage{eepic}

\numberwithin{equation}{section}

\newcommand{\Z}{\mathbb{Z}}
\newcommand{\Q}{\mathbb{Q}}
\newcommand{\R}{\mathbb{R}}
\newcommand{\C}{\mathbb{C}}

\newcommand{\A}{\mathbb A}

\newcommand\op\operatorname
\newcommand{\mi}{\,\mathrm{i}}

\newcommand{\cl}{\mathcal J} 
\newcommand{\dist}{d}
\newcommand{\norm}[1]{N_K(\langle #1  \rangle)}
\newcommand{\MM}{\mathcal{M}}

\newtheorem{theorem}{Theorem}[section]

\newtheorem{lemma}[theorem]{Lemma}
\newtheorem{proposition}[theorem]{Proposition}
\theoremstyle{definition}
\newtheorem{definition}[theorem]{Definition}
\theoremstyle{remark}
\newtheorem{remark}[theorem]{Remark}

\title[Projective Hermite constant]{A generalized Hermite constant for imaginary quadratic fields}
\author{Wai Kiu Chan, Mar\'ia In\'es Icaza, \and Emilio A. Lauret}

\address{Wai Kiu Chan, Department of Mathematics and Computer science, Wesleyan University, Middletown, CT 06459--0128, USA.}
\email{wkchan@wesleyan.edu}
\address{Mar\'ia In\'es Icaza, Instituto de Matem\'atica y F\'isica, Universidad de Talca, Casilla 747, Talca, Chile.}
\email{icazap@inst-mat.utalca.cl}
\address{Emilio Agust\'{\i}n Lauret, FaMAF-CIEM, Universidad Nacional de C\'ordoba, Ciudad Universitaria 5000, C\'ordoba, Argentina.}
\email{elauret@famaf.unc.edu.ar}



\begin{document}

\begin{abstract}
We introduce the \emph{projective Hermite constant} for positive definite binary hermitian forms associated with an imaginary quadratic number field $K$.
It is a lower bound for the classical Hermite constant, and these two constants coincide when $K$ has class number one.  Using the geometric tools developed by Mendoza \cite{Mendoza} and Vogtmann \cite{Vogtmann} for their study of the homology of the Bianchi groups, we compute the projective Hermite constants for those $K$ whose absolute discriminants are less than 70, and determine the hermitian forms that attain the projective Hermite constants in these cases.  A comparison of the projective hermitian constant with some other generalizations of the classical Hermite constant is also given.
\end{abstract}

\maketitle

\section{Introduction}

Let $K$ be an imaginary quadratic field and $\mathcal O_K$ be its ring of integers.  We regard $K$ as a subfield of $\C$ by identifying $K$ with its image under an embedding into $\C$.  For any binary hermitian form $S$ and any column vector $v \in \C^2$, $S(v)$ denotes the value of $S$ at $v$.  The (2-dimensional) Hermite constant $\gamma_K$ for $K$ is defined as
\begin{equation}\label{eq1:gamma_K}
\gamma_K = \max_{S \in \mathcal P} \min_{v \in \mathcal O_K^2\setminus \{0\}}\frac{ S(v)}{\det(S)^{1/2}},
\end{equation}
where $\mathcal P$ is the set of all positive definite binary hermitian forms.  In the subsequent discussion we often identify a binary hermitian form $ax\bar x + b\bar x y + \bar b x\bar y + cy\bar y$ with the $2\times 2$ hermitian matrix $\left(\begin{smallmatrix} a & b\\ \bar b& c \end{smallmatrix}\right)$, and  the determinant of a hermitian form will be the determinant of its associated hermitian matrix.  The constant $\gamma_K$ has been studied by many authors, and we will give a brief summary of some of the earlier results later in this section.

From now on, all hermitian forms are assumed to be positive definite.  As an attempt to take the ideal class group of $K$ into consideration, we define the {\em projective minimum} of a binary hermitian form $S$ to be
$$\mu_K^p(S):= \min_{v\in \mathcal O_K^2\setminus\{0\}} \frac{S(v)}{N_K(\langle v \rangle)},$$
where $N_K(\langle v\rangle)$ denotes the norm of the ideal generated by the coordinates of $v$.  As an analogy of \eqref{eq1:gamma_K}, we define the \emph{projective Hermite constant} $\gamma_K^p$ of $K$ by
\begin{equation}\label{eq1:gamma_K^p}
  \gamma_K^p: = \max_{S\in\mathcal P} \frac{\mu_K(S)}{\det(S)^{1/2}} = \max_{S\in\mathcal P} \min_{v\in \mathcal O_K^2\setminus\{0\}} \frac{S(v)}{N_K(\langle v \rangle)\det(S)^{1/2}}.
\end{equation}
It is clear that in both \eqref{eq1:gamma_K} and \eqref{eq1:gamma_K^p}, the maximum can be taken over only those $S$ with determinant 1.

It is not hard to see that $\gamma_K^p\leq\gamma_{K}$, and the equality holds when the class number $h_K$ of $K$ is one.   However, these two constants could be different when $h_K$ is greater than one, and $\gamma_K^p$ becomes a nontrivial lower bound for $\gamma_{K}$ in these cases.

A binary hermitian form $S$ is called \emph{projective extreme} (with respect to $K$) if $\gamma_K^p$ attains a local maximum at $S$, and is called \emph{absolutely projective extreme} if the maximum is absolute.

Suppose that $K = \Q(\sqrt{-D})$, where $D$ is a square-free positive integer.   An integral basis for $\mathcal O_K$ is given by $\{1, \omega\}$, where $\omega = \sqrt{-D}$ or $(1 + \sqrt{-D})/2$ if $D \equiv 1,2$ mod 4 or $D \equiv 3$ mod 4 accordingly.  We denote the discriminant of $K$ by $d_K$.  The constant $\gamma_K$ has been already determined for many $D$.  The case $D=1$ was due to Speiser~\cite{Speiser}, which was extended by Perron~\cite{Perron} to the cases $D=1,2,3,7,11,19$.  Oberseider~\cite{Oberseider} treated all imaginary quadratic fields of class number one by completing the cases $D=43,67,163$, and he completed the additional cases $D=5,6,10,13,15$ (class number 2) and $D=14,17$ (class number 4).

On the other hand, since $\mathcal O_K$ is a free $\mathbb Z$-module of rank two, a binary hermitian form over $K$ of determinant $\Delta$ can be viewed as a quaternary quadratic form over $\mathbb Q$ of determinant $(\Delta\vert d_K\vert)^2/16$.  This idea was exploited by Oppenheim~\cite{Oppenheim} to show that
\begin{equation} \label{eq1:Oppenheim-inequality}
\gamma_{K} \leq \sqrt{\frac{\vert d_K\vert}{2}},
\end{equation}
and that the equality holds if and only if $-2$ is a quadratic residue modulo $D$ or, equivalently, every odd prime divisor of $D$ is congruent to 1 or 3 modulo 8.

Since $\gamma_K^p = \gamma_{K}$ for every $K$ of class number one, $\gamma_K^p$ is determined for the nine imaginary quadratic number fields of class number one.  Our main result is the computation of $\gamma_K^p$ and the corresponding projective extreme forms for every imaginary quadratic field $K$ with class number not one and $|d_K|<70$.

\begin{theorem}\label{thm1:ctes}
Let $K$ be an imaginary quadratic number field whose absolute discriminant is less than $70$.
Then $\gamma_K^p$ and the associated absolutely projective extreme forms of determinant $1$, modulo the action of $\mathrm{SL}_2(\mathcal O_K)$, are given in Table~\ref{table1:Hermite-ctes} and Table~\ref{table critical points}, respectively.
\end{theorem}

In Table~\ref{table1:Hermite-ctes}, we also record the value of $\gamma_K$ whenever it is known for the sake of comparison.  For the convenience of presentation, in Table~\ref{table critical points} we use $[a,b,c]$ to denote the hermitian form $\left(\begin{smallmatrix} a & b\\\bar b & c\end{smallmatrix} \right)$.  Every absolutely projective extreme form in that table attains its projective minimum at $\left(\begin{smallmatrix}1\\0\end{smallmatrix}\right)$, except when it has a column vector $v$ as a subscript.  In that case, the projective minimum of that form is attained at $v$.


\begin{table}[!htb]
\caption{Projective Hermite constants for $\vert d_K \vert <70$}
\label{table1:Hermite-ctes}
$
\begin{array}{|ccccc|}\hline
-d_K& D & h_K& \gamma_K^p & \gamma_{K}    \\ \hline 
\rule{0pt}{12pt}
3   &3  &1  &\sqrt{3/2}    & \sqrt{3/2}        \\
4   &1  &1  &\sqrt2      & \sqrt2          \\
7   &7  &1  &\sqrt{7/3}    &\sqrt{7/3}        \\
8   &2  &1  & 2     & 2         \\
11  &11 &1  &\sqrt{11/2}   & \sqrt{11/2}       \\
15  &15 &2  &\sqrt3      &2\sqrt{5/3}       \\
19  &19 &1  &\sqrt{19/2}   &\sqrt{19/2}       \\
20  &5  &2  &\sqrt5      &4\sqrt{5/11}     \\
23  &23 &3  &\sqrt{23/5}   &           \\
24  &6  &2  &2\sqrt3    &2\sqrt3        \\
31  &31 &3  &\sqrt{31/3}   &           \\
35  &35 &2  &\sqrt7      &           \\ \hline
\end{array} \qquad
\begin{array}{|ccccc|}\hline
-d_K& D & h_K& \gamma_K^p & \gamma_{K}    \\ \hline
\rule{0pt}{12pt}
39  &39 &4  &\sqrt{13}     &          \\
40  &10 &2  &6\sqrt{5/13} &8\sqrt{10/39}    \\
43  &43 &1  &\sqrt{43/2}   & \sqrt{43/2}      \\
47  &47 &5  &\sqrt{47/5}  &          \\
51  &51 &2  &\sqrt{51/2}   & \sqrt{51/2}     \\
52  &13 &2  &\sqrt{52/3}   &8\sqrt{3/35}    \\
55  &55 &4  &4\sqrt{11/19} &          \\
56  &14 &4  &2\sqrt{7/3}   &8\sqrt{14/41}   \\
59  &59 &3  &\sqrt{59/2}   & \sqrt{59/2}      \\
67  &67 &1  &\sqrt{67/2}   & \sqrt{67/2}      \\
68  &17 &4  &\sqrt{34}     &\sqrt{34}        \\
163 &163&1  &\sqrt{163/2}  & \sqrt{163/2}     \\ \hline
\end{array}
$
\end{table}
\smallskip

\begin{table}[!htb]
\caption{Absolutely projective extreme forms  for $\vert d_K\vert < 70$}%
\label{table critical points}%
$
\begin{array}[t]{|ccc|}\hline
-d_K&h_K & \text{abs.\ proj.\ extreme forms} \\ \hline
\rule{0pt}{12pt}
3&1&
  [\frac{\sqrt{2}}{\sqrt{3}}, \,\pm \frac{i}{\sqrt{2}}, \,\frac{\sqrt{2}}{\sqrt{3}}]\\ \hline
4&1&
  [\sqrt{2},\, \pm \frac{1+\mi}{\sqrt{2}}, \,\sqrt{2}]\\ \hline
7&1&
  [\frac{\sqrt7}{\sqrt3}, \,\pm\frac{2\mi}{\sqrt3}, \,\frac{\sqrt7}{\sqrt3}]\\ \hline
8&1&
 [2, \,\pm\frac{\sqrt2 + 2\mi}{\sqrt2}, \,2]\\ \hline
11&1&
  [\frac{\sqrt{11}}{\sqrt2},\, \pm\frac{3\mi}{\sqrt2}, \,\frac{\sqrt{11}}{\sqrt2}]\\ \hline
15&2& [\sqrt3, \, \pm \sqrt5 \mi, \, 2\sqrt3]\\ \hline
19&1& [\frac{\sqrt{19}}{\sqrt2},\, \pm \frac{6\mi}{\sqrt2}, \, \frac{2\sqrt{19}}{\sqrt2}] \\ \hline
20&2& [\sqrt5, \, \pm \sqrt5 \mi, \, 3\sqrt5] \\ \hline
23&3& [\frac{\sqrt{23}}{\sqrt5}, \, \pm\frac{\sqrt8\mi}{\sqrt5},\, \frac{\sqrt{23}}{\sqrt5}]\\ \hline
24&2& [\sqrt{12}, \, \pm\frac{\sqrt6 + 4\mi}{\sqrt2},\, \sqrt{12}]\\ \hline
31&3& [\frac{\sqrt{31}}{\sqrt3}, \, \pm\frac{\sqrt8 \mi}{\sqrt5}, \, \frac{4\sqrt{31}}{\sqrt3}]\\ \hline
35&2& [\sqrt7, \, \pm\frac{10 \mi}{\sqrt5},\, 3\sqrt7]\\ \hline
39&4& [\sqrt{13},\, \pm\frac{6\mi}{\sqrt3},\, \sqrt{13}]\\ \hline
40&2&
  \begin{array}{c}
    \left[\frac{3\sqrt{20}}{\sqrt{13}},\, -\frac{3\sqrt{10}\pm 22\mi}{\sqrt{26}},\, \frac{10\sqrt{10}}{\sqrt{13}}\right] \\
    \left[\frac{3\sqrt{20}}{\sqrt{13}},\, \pm\frac{\sqrt{10}\pm 18\mi}{\sqrt{26}},\, \frac{3\sqrt{20}}{\sqrt{13}}\right]
  \end{array}\\ \hline
43&1&
  [\frac{\sqrt{43}}{\sqrt2},\, \pm \frac{16\mi}{\sqrt2},\, \frac{6\sqrt{43}}{\sqrt2}]\\ \hline
47&5&
  \begin{array}{c}
    \left[\frac{\sqrt{47}}{\sqrt5},\, \pm\frac{18\mi}{\sqrt5},\, \frac{7\sqrt{47}}{\sqrt5}\right]\\
    \left[\frac{2\sqrt{47}}{\sqrt5},\, \pm\frac{29\mi}{\sqrt5},\, \frac{9\sqrt{47}}{\sqrt5}\right]_{\left(\begin{smallmatrix}\omega\\2\end{smallmatrix}\right)}
  \end{array}\\ \hline
\end{array}\hskip 2mm
\begin{array}[t]{|ccc|}\hline
-d_K&h_K & \text{abs.\ proj.\ extreme forms} \\ \hline
\rule{0pt}{12pt}
51&2&
  \begin{array}{c}
    \left[\frac{\sqrt{51}}{\sqrt2},\, \pm\frac{7\mi}{\sqrt2},\, \frac{\sqrt{51}}{\sqrt2}\right]\\
    \left[\frac{\sqrt{51}}{\sqrt2},\, \pm\frac{10\mi}{\sqrt2},\, \frac{2\sqrt{51}}{\sqrt2}\right]
  \end{array}\\ \hline
52&2&
  \begin{array}{c}
    \left[\frac{\sqrt{52}}{\sqrt2},\, \pm\frac{7\mi}{\sqrt3},\, \frac{\sqrt{52}}{\sqrt2}\right]\\
    \left[\frac{\sqrt{52}}{\sqrt2},\, -\frac{\sqrt{52}\pm 6\mi}{\sqrt3},\, \frac{\sqrt{52}}{\sqrt2}\right]
  \end{array}\\ \hline
55&4&
  \begin{array}{c}
    \left[\frac{4\sqrt{11}}{\sqrt{19}},\, \pm\frac{35\mi}{\sqrt{95}},\, \frac{6\sqrt{11}}{\sqrt{19}}\right]\\
    \left[\frac{5\sqrt{11}}{\sqrt{19}},\, \pm\frac{7\sqrt{11}+5\sqrt{19}\mi}{2\sqrt{19}},\, \frac{11\sqrt{11}}{\sqrt{19}}\right]\\
    \left[\frac{5\sqrt{11}}{\sqrt{19}},\, \pm\left(\frac{3\sqrt{11}}{2\sqrt{19}} + \frac{\sqrt{19}\mi}{\sqrt{5}}\right),\, \frac{9\sqrt{11}}{\sqrt{19}}\right]
  \end{array}\\ \hline
56&4&
  \begin{array}{c}
    \left[\frac{\sqrt{28}}{\sqrt3},\, -\frac{\sqrt7 + 6\sqrt2\mi}{\sqrt3},\, \frac{79}{\sqrt{84}}\right]\\
    \left[\frac{\sqrt{28}}{\sqrt3},\, -\frac{11\sqrt7 + 53\frac2\mi}{\sqrt3},\, \frac{2155\sqrt3}{11\sqrt{28}}\right]_{\left(\begin{smallmatrix}1 + \omega\\3\end{smallmatrix}\right)}
  \end{array}\\ \hline
59&3&
  \begin{array}{c}
    \left[\frac{\sqrt{59}}{\sqrt2},\, \pm\frac{23\mi}{\sqrt2},\, \frac{9\sqrt{59}}{\sqrt2}\right]\\
    \left[\sqrt{108},\, -\frac{59+23\sqrt{59}\mi}{\sqrt{108}},\, \frac{5\sqrt{59}}{\sqrt2}\right]_{\left(\begin{smallmatrix}2 + \omega\\3\end{smallmatrix}\right)}
  \end{array}\\ \hline
67&1& [\frac{\sqrt{67}}{\sqrt2},\, \pm\frac{20\mi}{\sqrt2\mi},\, \frac{6\sqrt{67}}{\sqrt2}]\\ \hline
68&4&
  \begin{array}{c}
    \left[\sqrt{34},\, -\frac{\sqrt{17}+7\mi}{\sqrt2},\, \sqrt{34}\right]\\
    \left[9\sqrt{34},\, -\frac{9\sqrt{17}+95\mi}{\sqrt2},\, 17\sqrt{34}\right]_{\left(\begin{smallmatrix}1 + \omega\\3\end{smallmatrix}\right)}
  \end{array}\\ \hline
163&1& [\frac{\sqrt{163}}{\sqrt2},\, \pm\frac{18\mi}{\sqrt2},\, \frac{2\sqrt{163}}{\sqrt2}]
\\ \hline
\end{array}
$
\end{table}

\begin{remark}
Our computation shows that when $\vert d_K \vert < 70$, $\gamma_{K}$ attains the upper bound $\sqrt{{\vert d_K\vert}/{2}}$ if and only if $\gamma_{K}^p$ is also equal to $\sqrt{{\vert d_K\vert}/{2}}$ (there are twelve of these $K$).
Although our sample size is small, the results we have here do lead to the following interesting question:
\begin{center}
  \emph{If $\gamma_{K} = \sqrt{\dfrac{\vert d_K \vert}2}$, is $\gamma_K^p$ also equal to $\sqrt{\dfrac{\vert d_K\vert}2}$?}
\end{center}
\end{remark}

The proof of Theorem~\ref{thm1:ctes} relies heavily on the geometric tools developed by Mendoza~\cite{Mendoza} and Vogtmann~\cite{Vogtmann} for their study of the homology of the Bianchi groups.  Our idea is to make use of the classical bijection between the binary hermitian forms of a fixed determinant and the upper-half space $\mathrm H^3$, which is equivariant under the action of the Bianchi group $\Gamma = \mathrm{SL}_2(\mathcal O_K)$ (see Theorem~\ref{Thm bijection H3 and hermitian forms}).  Mendoza~\cite[Definition 2.1.6]{Mendoza} constructs a set $\MM$, which is a $2$-dimensional cellular retract of $\mathrm H^3$ and is invariant under $\Gamma$.  His result \cite[Page~30]{Mendoza} implies that the function $S \mapsto \mu_K^p(S)/\det(S)^{1/2}$
attains its local maxima at the vertices of $\MM$  \cite[Page~30]{Mendoza}.  Our main task is to apply Vogtmann's theorem to build an appropriate cellular fundamental domain of $\MM$ under the action of $\Gamma$ and to recover its vertices.

The paper is organized as follows.  In Section \ref{sec2:geometric-tools} we introduce the well-known geometric relationship between the set of binary hermitian forms of a fixed determinant and the upper-half space, and we recall some results in \cite{Mendoza} by Mendoza.  Section \ref{sec3:Vogtmann} describes Vogtmann's theorem \cite{Vogtmann}; it determines a convenient cellular fundamental domain which will be used for the computation of $\gamma_K^p$ and the projective extreme forms. In Section \ref{sec4:Computations} we discuss the idea behind the computations needed in the proof of Theorem~\ref{thm1:ctes}, and illustrate it
by working out in detail the particular case $K=\Q(\sqrt{-39})$ which has class number four and whose calculation is already quite involved.  In Section~\ref{sec5:comparison}, we define the $n$-dimensional \emph{projective Hermite-Humbert} constant for an arbitrary number field $K$ and any $n\geq2$, which can be viewed as a generalization of $\gamma_K^p$.  A comparison of this constant with other generalization of the Hermite constant given by Icaza~\cite{I}, Thunder~\cite{Thunder}, and Watanabe~\cite{W-Lie, W-Fund} will be given.

\section{Geometric tools} \label{sec2:geometric-tools}

The group $\mathrm{GL}_2(\C)$ acts on $\mathcal P$ by
\begin{equation*}
S\mapsto g\cdot S= |\det g|\; (g^{-1})^{*}\;S\; (g^{-1}),
\end{equation*}
where $*$ denotes the conjugate transpose of a matrix.
The set $\mathcal P(\Delta)$ of binary hermitian forms with a fixed determinant $\Delta$ is invariant under this action, for every $\Delta>0$.
Let $\mathrm H^3$ denote the $3$-dimensional (real) hyperbolic space which is realized as the upper half-space of $\C\times\R$, that is,
$$
\mathrm H^3=\{(z,\zeta)\in\C\times\R\;:\;\zeta>0\}.
$$
There is a natural action of $\mathrm{GL}_2(\C)$ on $\mathrm H^3$ given by
\begin{equation*}
  g\cdot(z,\zeta)=\left(
    \frac{(az+b)\overline{(cz+d)}+a\bar c\zeta^2}{|cz+d|^2+|c|^2\zeta^2},
    \frac{|\det g|\;\zeta }{|cz+d|^2+|c|^2\zeta^2}\right).
\end{equation*}
The subgroup $\mathrm{SL}_2(\C)$ of $\mathrm{GL}_2(\C)$ acts on $\mathrm H^3$ as isometries.  In fact, $\mathrm{PSL}_2(\C)$ is the full group of orientation preserving isometries of $\mathrm H^3$.
The next theorem relates these two actions.

\begin{theorem}\cite[Proposition 1.1]{EGMzeta}\label{Thm bijection H3 and hermitian forms}
The map $\Phi:\mathcal P\to\mathrm H^3$ defined by
$$
\Phi\left(\begin{matrix}a&b\\\bar b&c\end{matrix}\right)=\left(\frac{-b}{a},\frac{\sqrt{ac - \vert b\vert^2}}{a}\right)
$$
is $\mathrm{GL}_2(\C)$-equivariant, and $\Phi_\Delta:=\Phi|_{\mathcal P(\Delta)}:\mathcal P(\Delta)\to\mathrm H^3$ is a bijection for all $\Delta>0$. The inverse of this map is given by
$$
\Psi_\Delta(z,\zeta)=\frac{\sqrt{\Delta}}{\zeta}
  \begin{pmatrix} 1& -z\\ -\bar z&|z|^2+\zeta^2 \end{pmatrix}.
$$
\end{theorem}

The action of $\mathrm{GL}_2(\C)$ on $\mathrm H^3$ extends to the boundary of $\mathrm H^3$. We identify this boundary with $\mathbb P^1(\C)=\C\cup\{\infty\}$, where the elements in $\C$ correspond to the elements in $\C\times\R$ with second component equal to zero.  Via this identification, the action of $\mathrm{GL}_2(\C)$ on this boundary is precisely the classical M\"obius action.

An element $\lambda\in\mathbb P^1(\C)$ is called a \emph{cusp} if either  $\lambda\in K\subset\C$ or $\lambda = \infty$. The set of cusps can be identified with $\mathbb P^1(K)$ and, from now on, we will write  $\lambda=\alpha/\beta$ with $\alpha,\beta\in\mathcal O_K$ if $\lambda\in K$, or $\lambda=1/0$ if $\lambda=\infty$.  Note that the subgroup $\mathrm{GL}_2(K)$ preserves the set of all cusps.  Following the classical theory we associate to each cusp $\lambda=\alpha/\beta$ the element of the ideal class group $\cl_K$ containing the fractional ideal $\langle\alpha,\beta\rangle$.   It is easy to see that this map is well-defined and it induces a bijection between $\mathrm{SL}_2(\mathcal O_K)\backslash \mathbb P^1(K)$ and $\cl_K$ (see \cite[\S 7.2]{EGMbook}).

We are now ready to describe Mendoza's work \cite{Mendoza}.
The \emph{distance between a point $P=(z,\zeta)\in\mathrm{H}^3$ and a cusp $\lambda$} is given by
\begin{equation*}
\dist(P,\lambda)=\frac{|\beta z-\alpha|^2+|\beta|^2\zeta^2}{\zeta \, \norm{\alpha,\beta}}.
\end{equation*}
This definition does not depend on the choices of $\alpha$ and $\beta$, as we shall see in the following remarks.

\begin{remark}[geometric interpretation]\label{rmk2:geom-interpretation}
It follows immediately that for any $\lambda\in K$, $N_\lambda := \norm{\alpha,\beta}/|\beta|^2$ does not depend on the choices of the integers $\alpha$ and $\beta$ as long as $\lambda =\alpha/\beta\in K$.
This gives $\dist(P,\lambda)={|P-\lambda|^2}/{(\zeta\, N_\lambda)}$.
\end{remark}

If $\lambda$ and $\mu$ are two cusps, let $S(\lambda,\mu)$ be the set of points in $\mathrm H^3$ that are equidistant from these cusps. These sets are hemispheres or planes perpendicular to the boundary $\C$.  In other words, they are geodesic planes in $\mathrm H^3$. In particular, $S(\infty,\lambda)$ is a hemisphere centered at $\lambda$ with radius $\sqrt{N_\lambda}$.

\begin{remark}[arithmetic interpretation]\label{rmk2:arith-interpretation}
Write $S=\Psi_\Delta(P)\in \mathcal P(\Delta)$ and $v=\left(\begin{smallmatrix}\alpha\\ \beta\end{smallmatrix}\right) \in\mathcal O_K^2$ where $\lambda=\alpha/\beta$. Then
\begin{equation}\label{eq2:arith-interpretation}
\dist(P,\lambda) = \frac{S(v)}{\sqrt{\Delta}\;\norm{\alpha,\beta }}.
\end{equation}
\end{remark}

As an easy application of (\ref{eq2:arith-interpretation}) we have the following proposition, which can also be obtained from combining Propositions 3.2 and 3.3 of \cite{Vogtmann}.

\begin{proposition}\label{prop transformation rule}
If $g=\left(\begin{smallmatrix}a&b\\c&d\end{smallmatrix}\right)\in \mathrm{M}_2(\mathcal O_K)$ with $\det(g)\neq0$, $P\in\mathrm H^3$ and $\lambda=\alpha/\beta\in\mathbb{P}^1(K)$, then
  $$
  \dist(g\cdot P,g\cdot\lambda)=\frac{|\det g|\;\norm{\alpha,\beta}}{\norm{a\alpha+b\beta,c\alpha+d\beta}}\; \dist (P,\lambda).
  $$
In particular $\dist(g\cdot P,g\cdot\lambda)=\dist (P,\lambda)$ for all $g\in\mathrm{SL}_2(\mathcal O_K)$.
\end{proposition}

In view of \eqref{eq1:gamma_K^p} and \eqref{eq2:arith-interpretation}, the projective Hermite constant $\gamma_K^p$ can be computed by
\begin{equation}\label{eq2:gamma^p-H3}
\gamma_K^p=\sup_{P\in\mathrm H^3} \inf_{\lambda\in \mathbb P^1(K)} \dist(P,\lambda).
\end{equation}

A positive number $c\in\R$ is called an \emph{upper reduction constant} for $K$ if for each $P \in \mathrm H^3$ there is at least one cusp $\lambda$ of $K$ such that $d(P,\lambda)\leq c$. The \emph{optimal upper reduction constant} for $K$ is the infimum of all upper reduction constants for $K$.

\begin{theorem}\cite{Mendoza} \label{Mendoza}
For $P$ a point in $\mathrm H^3$, the following hold.
\begin{enumerate}
\item For any real number $c>0$, there are only finitely many cusps $\lambda$ of $K$ such that $\dist(P,\lambda)\leq c$.
\item $\sqrt{|d_K|/2}$ is an upper reduction constant for $K$.
\item Consider the function $\eta_K:\mathrm H^3\to\R$ defined by
\begin{equation}\label{eq eta}
\eta_K(P)=\inf_{\lambda\in\mathbb P^1(K)} \dist(P,\lambda).
\end{equation}
Then $\eta_K$ is continuous and invariant under the action of $\mathrm{SL}_2(\mathcal O_K)$.
\item The function $\eta_K$ attains a  maximum on $\mathrm H^3$, which is equal to the optimal upper reduction constant for $K$.
  \end{enumerate}
\end{theorem}

Note that Part~(2) of the above theorem can be seen from (\ref{eq1:Oppenheim-inequality}).

\begin{definition}\label{def H(lambda)}
Let $\lambda$ be a cusp of $K$. The \emph{minimal set $H(\lambda)$ of $\lambda$} is the set of points in $\mathrm H^3$ such that the distance from $\lambda$ is smaller than or equal to the distance from any other cusp, that is,
$$
H(\lambda)=\{P\in\mathrm H^3\;:\; \dist(P,\lambda)\leq\dist(P,\mu)\quad\text{for all $\mu\in \mathbb P^1(K)$}\}.
$$
\end{definition}

\begin{remark}\label{rmk H(infty)}
It follows from the definition that $\eta_K(P) = \dist(P, \lambda)$ for all $P$ in $H(\lambda)$.
The set $H(\infty)$ is easily seen to be the set of points in $\mathrm H^3$ which lie above all hemispheres $S(\infty,\lambda)$ for all finite cusps $\lambda$.
\end{remark}

We denote by $\Gamma$ the Bianchi group $\mathrm{SL}_2(\mathcal O_K)$, and by $\Gamma(\lambda)$ the stabilizer of $\lambda$ in $\Gamma$. Let  $\{\lambda_1,\lambda_2,\dots,\lambda_{h_K}\}$ be a set of representatives of the $\mathrm{SL}_2(\mathcal O_K)$-orbits in the set of cusps. From Proposition \ref{prop transformation rule} and Theorem \ref{Mendoza} one can easily see that $\Gamma\cdot \big(H(\lambda_1)\cup\dots\cup H(\lambda_{h_K})\big)=\mathrm H^3$.  Note that this set is far from being a fundamental domain of the action of $\Gamma$ on $\mathrm H^3$, since $\Gamma(\lambda_i)$ preserves $H(\lambda_i)$ for each $i$.

The \emph{minimal incidence set} \cite[Definition 2.1.6]{Mendoza} is
\begin{equation}\label{eq minimal incidence set}
\MM=\bigcup_\lambda \partial H(\lambda)=\bigcup_{\lambda\neq\mu} H(\lambda)\cap H(\mu),
\end{equation}
This $\MM$ has been used in the computation of integral and rational homology of Bianchi groups, since it is a $2$-dimensional cellular retract of $\mathrm H^3$ which is invariant under $\mathrm{SL}_2(\mathcal O_K)$. The readers are referred to \cite{Mendoza},  \cite{Vogtmann}, and the references therein for more information.

\begin{lemma}\cite[Proposition 2.1.7]{Mendoza} \label{mendoza2}
If the function $\eta_K$ defined in $(\ref{eq eta})$ attains its maximum at a point $P$, then $P$ is in $\MM$.
\end{lemma}

Now, \eqref{eq2:gamma^p-H3}, Proposition~\ref{prop transformation rule}, Theorem~\ref{Mendoza} and Lemma~\ref{mendoza2} allow us to write $\gamma_K^p$ as
\begin{equation}\label{eq2:1st-red-Mendoza}
\gamma_K^p=\max_{P\in \MM}\eta_K(P) = \max_{1\leq i\leq h_K}\max_{P\in \widetilde{T}_i} \;\dist(P,\lambda_i),
\end{equation}
where $\widetilde{T}_i$ denotes a fundamental domain of $\partial H(\lambda_i)$ under the action of $\Gamma(\lambda_i)$. In the next section we shall use the work of Vogtmann in \cite{Vogtmann} to construct the sets $\widetilde{T}_i$.

A point at which $\eta_K$ attains a local  maximum is called a \emph{projective extreme} point.
A point is called an \emph{absolutely projective extreme} point if $\eta_K$ attains the global maximum there.
It is explained in \cite[Page 30]{Mendoza} that the projective extreme points are the vertices of $\MM$.
By definition, projective extreme points and projective extreme forms of a fixed discriminant are in bijective correspondence.

\begin{proposition}\label{prop2:extreme-points}
The map $\Phi$ defined in Theorem $\ref{Thm bijection H3 and hermitian forms}$ identifies projective extreme forms (resp.\ absolutely projective extreme forms) of a fixed determinant with projective extreme points (resp.\ absolutely projective extreme points).
\end{proposition}

\section{Vogtmann's results}\label{sec3:Vogtmann}

This section contains a brief summary of results of Vogtmann in \cite{Vogtmann}.  The next theorem gives an effective method to find a cellular fundamental domain of $\Gamma(\lambda)\backslash \partial H(\lambda)$.  Given a finite cusp $\lambda$ of $K$, let $n$ be the smallest positive integer such that $\alpha=n\lambda\in\mathcal O_K$, and let
$$
L_\lambda = \begin{pmatrix}\alpha&1\\ n&0\end{pmatrix}.
$$
For $\lambda=\infty$, we set $\alpha=1$, $n=0$ and $L_\infty$ the identity matrix.  The matrix $L_\lambda$ will be regarded as an element in $\text{GL}_2(\C)$ and so acting on $\mathrm H^3$.

\begin{theorem}\cite[Theorem (4.9)]{Vogtmann}\label{thm3:Vogtmann}
Suppose that $|d_K|>4$.
Let $\lambda$ be a cusp of $K$, and $\mathcal D$ be a fundamental domain of the lattice $n\langle\alpha,n\rangle^{-2}$ (or $\mathcal O_K$ if $\lambda = \infty$) in $\C$.  A fundamental domain for the action of $\Gamma(\lambda)$ on $\partial H(\lambda)$ is given by the union of the cells $H(\lambda)\cap H(L_\lambda(\gamma/m))$ such that
\begin{enumerate}
\item[(i)] $m$ is a rational  integer satisfying $0<m<\norm{\alpha,n}(-d_K/2)$;
\item[(ii)] $\gamma\in\mathcal O_K$ and $\gamma/m\in\mathcal D$;
\item[(iii)] $\dim{\big(H(\lambda)\cap H(L_\lambda(\gamma/m))\big)}=2$.
\end{enumerate}
\end{theorem}

The fundamental domain for $\Gamma(\lambda)$ given in the above theorem will be denoted by $I(\lambda)$.

\begin{remark} \label{Iinfinity}
The case $\lambda = \infty$ is missing in the statement of \cite[Theorem (4.9)]{Vogtmann}.   However, the assumption $\vert d_K\vert > 4$ implies that $\Gamma(\infty)$ is the group of translations of $\C$ by elements in $\mathcal O_K$.  As a result, $\mathcal D$ can be chosen to be a fundamental domain of $\mathcal O_K$ in $\C$ when $\lambda = \infty$.
\end{remark}

\begin{remark}
When $|d_K| = -4$ and $-3$, $\Gamma(\infty)$ contains isometries other than the translations by elements of $\mathcal O_K$, namely the rotations by $\pi$ and $2\pi/3$ respectively.  In these two cases, we obtain a fundamental domain of $\Gamma(\lambda)$ on $\partial H(\lambda)$ by first applying the theorem and then taking the quotient of the resulting space by the rotations.
\end{remark}

When $K$ has class number one, there is only one $\mathrm{SL}_2(\mathcal O_K)$-orbit of cusps. To compute $\gamma_K^p$ in this case, it follows from \eqref{eq2:1st-red-Mendoza}  that it is sufficient to determine a cellular fundamental domain for $\Gamma(\infty)\backslash \MM$, and Theorem \ref{thm3:Vogtmann} tells us how to do exactly that.  Moreover, for determining the projective extreme points, we have to locate the vertices of this fundamental domain, modulo the action of $\Gamma(\infty)$.

In the general case when the class number $h_K$ of $K$ is not one, we should determine the $h_K$  fundamental domains $\Gamma(\lambda_i)\backslash \partial H(\lambda_i)$ for each $\lambda_i$. If $\lambda\neq\infty$, the minimal set $H(\lambda)$ is not as nice as $H(\infty)$, graphically speaking, because $H(\lambda)$ cannot be seen directly from the top.  For these $\lambda$, we will follow the proof of \cite[Theorem (4.9)]{Vogtmann} to construct a cellular fundamental domain for a conjugate of $\Gamma(\lambda)$ which is visible directly from the top, and then apply an isometry to obtain a cellular fundamental domain of $\Gamma(\lambda)$ on $\partial H(\lambda)$.  We briefly explain this procedure below.

The \emph{$\lambda$-distance} from a point $P\in\mathrm H^3$ to a cusp $\mu$ is defined to be
\begin{equation*}
\dist_\lambda(P,\mu)= \dist(L_\lambda(P),L_\lambda(\mu)).
\end{equation*}
For $\mu,\nu\in\mathbb P^1(K)$ we define
\begin{align*}
  S_\lambda(\mu,\nu)&=\{P\in\mathrm H^3\;:\; \dist_\lambda(P,\mu)=\dist_\lambda(P,\nu)\},\text{ and}\\
  H_\lambda(\mu)&=\{P\in\mathrm H^3\;:\; \dist_\lambda(P,\mu)\leq\dist_\lambda(P,\nu)\quad\forall\,\nu\in\mathbb P^1(K)\}.
\end{align*}
The set $H_\lambda(\mu)$ is called the \emph{$\lambda$-minimal set of $\mu$}.

\begin{remark}\label{rmk H_lambda(infty)}
As in Remark \ref{rmk H(infty)}, $S_\lambda(\infty,\mu)$ is a hemisphere centered at $\mu=\gamma/\delta$ with radius the square root of ${\norm{L_\lambda(\gamma,\delta)}}/({\norm{\alpha,n}|\delta|^2})$. Furthermore, $H_\lambda(\infty)$ is the set of points in $\mathrm H^3$ which lie above all hemispheres $S_\lambda(\infty,\mu)$ for every finite cusp $\mu$.  Therefore, $H_\lambda(\infty)$ is visible directly from the top.
\end{remark}

\begin{remark} \label{Ilambda}
Note that
\begin{equation}\label{eq L(H_lambda(mu))=H(L(mu))}
  L_\lambda(H_\lambda(\mu))=H(L_\lambda(\mu)).
\end{equation}
Since $L_\lambda(\infty) = \lambda$, we have $L_\lambda(H_\lambda(\infty)) = H(\lambda)$; thus the cell structure of $\partial H(\lambda)$ is the same as that of $\partial H_\lambda(\infty)$.   This means that if $I_\lambda$ is a fundamental domain for $L_\lambda^{-1}\Gamma(\lambda)L_\lambda$ on $\partial H_\lambda(\infty)$, then $L_\lambda(I_\lambda)$ will be a fundamental domain for $\Gamma(\lambda)$ on $\partial H(\lambda)$.  Typically we will take $I_\lambda$ to be the union of the sets $H_\lambda(\infty)\cap H_\lambda(\gamma/m)$ over the cusps $\gamma/m$ which satisfies (i), (ii) and (iii) in Theorem \ref{thm3:Vogtmann}. Note that $L_\lambda^{-1}(H(\lambda)\cap H(L_\lambda(\gamma/m)))=H_\lambda(\infty)\cap H_\lambda(\gamma/m)$ by \eqref{eq L(H_lambda(mu))=H(L(mu))}, and so this $I_\lambda$ is indeed a fundamental domain for $L_\lambda^{-1}\Gamma(\lambda)L_\lambda$ on $\partial H_\lambda(\infty)$.
\end{remark}

In some occasions, however, we can exploit the symmetries between the cusps and obtain $I_\lambda$ in a more conveniently manner.   For example, using the equality  $\dist_{\bar\lambda}((\bar z,\zeta), \bar\mu)=\dist_{\lambda}((z,\zeta), \mu)$, one obtains the following lemma.

\begin{lemma}\cite[Corollary (5.5)]{Vogtmann} \label{lem H_lambda and H_bar-lambda}
  $H_\lambda(\infty)$ is the reflection of $H_{\bar\lambda}(\infty)$ through the plane $\mathrm{Im}(z)=0$.
\end{lemma}

For $\mu=\gamma/\delta\in K$, let $[\mu]$ be the class of the ideal $\langle\gamma,\delta\rangle$.  One finds a useful application of Lemma
\ref{lem H_lambda and H_bar-lambda} when the ideal class group $\cl_K$ is $\{[\infty],[\lambda],[\bar\lambda]\}$ for some cusp $\lambda$.  In this case, finding $\gamma_K^p$ requires only the maximum of $\eta_K$ in $\partial H(\infty)$ and in $\partial H(\lambda)$.  A less obvious type of symmetry follows from the following theorem which is essentially Proposition 5.2 and Theorem 5.6 in \cite{Vogtmann} together.  We provide a sketch of its proof since certain features of it  will be used in the next section.

\begin{theorem}\cite{Vogtmann}\label{thm class number two Vogtmann}
Let $\lambda$ be a cusp of $K$, and suppose that $[\lambda]^2=1$. If $\mu$ and $\nu$ are cusps of $K$ such that $[\mu][\lambda]=[\nu]$, then there exists $h\in\mathrm{GL}_2(\mathcal O_K)$ such that $h(\mu) = \nu$ and $h$ is an isomorphism sending $H(\mu)$ to $H(\nu)$.  In particular, $\partial H(\mu)$ is isomorphic to $\partial H(\nu)$.
\end{theorem}

\begin{proof}[Sketch of proof]
Write $\lambda=\alpha/\beta$. Since $[\lambda]^2=1$, one can prove that there is a matrix
\begin{equation}\label{eq g class number two}
g=\left(\begin{matrix}x&y\\z&w\end{matrix}\right)\in\mathrm{GL}_2(\mathcal O_K)
\quad\text{such that}\quad
  \begin{cases}
    \langle\alpha,\beta\rangle= \langle x,y\rangle= \langle z,w\rangle,\\
    |\det (g)|=\norm{\alpha,\beta},
  \end{cases}
\end{equation}
inducing an isomorphism $g:\mathcal O_K\oplus\mathcal O_K\longrightarrow \langle \alpha,\beta\rangle\oplus \langle \alpha,\beta\rangle$. It follows that $\langle g(s,t)\rangle= \langle xs+yt,zs+wt\rangle =\langle s,t\rangle\langle \alpha,\beta\rangle$ for all $(s,t)\in\mathcal O_K^2$, and hence $[g\cdot\mu]=[\mu][\lambda]=[\nu]$ and $\norm{g(s,t)}=\norm{ s,t}\norm{\alpha,\beta}$. Writing $\mu=\gamma/\delta$ and applying Proposition~\ref{prop transformation rule}, it follows that
\begin{align*}
\dist(g\cdot P,g\cdot\mu)=& \left(\frac{|\det g| \norm{\gamma,\delta}}{\norm{g(\gamma,\delta)}}\right) \dist(P,\mu)=\dist(P,\mu),
\end{align*}
and therefore $g$ is an isomorphism taking $H(\mu)$ to $H(g\cdot\mu)$.  Since $[g\cdot \mu] = [\nu]$, there exists $g' \in \mathrm{SL}_2(\mathcal O_K)$ such that $g'g(\mu) = \nu$.  Thus $h: = g'g$ is an isomorphism sending $H(\mu)$ to $H(\nu)$.
\end{proof}

\begin{remark}\label{rmk using thm class number two}
Assume that $[\lambda]^2 = 1$.
Then we can apply the above theorem to obtain an $h \in \mathrm{GL}_2(\mathcal O_K)$ such that $h(\infty)=\lambda$.
Then $h(I(\infty))$ is a fundamental domain for $\Gamma(\lambda)\backslash \partial H(\lambda)$.
Therefore,
\begin{equation*}
\max_{P\in \partial H(\lambda)} \dist(P,\lambda) = \max_{P\in I(\infty)} \dist(g\cdot P,g\cdot\infty) = \max_{P\in I(\infty)} \dist(P,\infty).
\end{equation*}
Hence there is no need to consider the maximum of $\eta_K$ on $\partial H(\lambda)$ in (\ref{eq2:1st-red-Mendoza}).
Similarly, if $[\nu]=[\lambda][\mu]$, then the maximum of $\eta_K$ on $\partial H(\nu)$ is equal to its maximum on $\partial H(\mu)$.
\end{remark}

We conclude this section with a lemma from algebraic number theory which will help us bound the size of the hemispheres $S_\lambda(\mu,\nu)$.

\begin{lemma}\cite[Lemma 4.3]{Vogtmann}\label{lem 4.3 Vogtmann}
Let $\lambda$ be a cusp of $K$, and $n$ be the smallest positive integer such that $n\lambda \in \mathcal O_K$.  Then, for any finite cusp $\mu$ of $K$,
$$
\norm{L_\lambda(\gamma,m)}=\norm{\alpha\gamma+m,n\gamma}\leq
\begin{cases}
  m& \mbox{ if $\lambda=\infty$};\\
  n^2 m& \mbox{ if $\lambda\neq\infty$},
\end{cases}
$$
where $m$ is the smallest positive integer such that $m\mu\in\mathcal O_K$.
\end{lemma}

\section{Computations}\label{sec4:Computations}

We have already laid down all necessary background to compute $\gamma_K^p$ and to determine the projective extreme forms for an arbitrary imaginary quadratic number field $K$.  Our method for computing $\gamma_K^p$ can be summarized in the following steps:

\begin{enumerate}
  \item[(I)] Find cusps $\lambda_1,\dots,\lambda_{h_K}$ such that $[\lambda_1], \ldots, [\lambda_{h_K}]$ are all the elements in the class group of $K$.
  \item[(II)] For each $1\leq i\leq h_K$, determine the set $I(\lambda_i)$  using Theorem~\ref{thm3:Vogtmann}.
  \item[(III)] Find the minimum of $\eta_K$ on the set of vertices of each $I(\lambda_i)$.
\end{enumerate}

Step (I) is the easiest, and Step (III) will follow immediately once we construct $I(\lambda_i)$ in Step (II).  For this step, we will use graphical aid (e.g. Sage) to help us produce a candidate $J_{\lambda_i}$ for $I_{\lambda_i}$.  Once we prove that $J_{\lambda_i}$ is indeed  $I_{\lambda_i}$, we may apply $L_{\lambda_i}$ to obtain $I(\lambda_i)$ (see Remark~\ref{Ilambda}).  This will be further explained in (II.1) to (II.4) below.

In what follows, if we write a cusp $\mu$ as $\delta/k$, then $k$ is the smallest positive integer such that $k\mu\in\mathcal O_K$.    Fix a cusp $\lambda = \alpha/n$, and let $\{\pi_1, \pi_2\}$ be an integral basis for the ideal $n\langle \alpha, n \rangle^{-2}$ (or $\mathcal O_K$ if $\lambda = \infty$). If $\lambda = \infty$, we choose $\pi_1$ and $\pi_2$ to be 1 and $\omega$ respectively.  Recall that $\omega = \sqrt{-D}$ or $(1 + \sqrt{-D})/2$ if $D \equiv 1, 2$ mod 4 or $D \equiv 3$ mod 4 accordingly.  Let $\mathcal D$ be $\{x\pi_1 + y\pi_2: 0 \leq x, y < 1\}$.

\begin{enumerate}
\item[(II.1)] Use Sage to draw all hemispheres $S_\lambda(\gamma/m, \infty)$, with $0 < m < -N(\langle \alpha, n \rangle)\frac{d_K}{2}$ and $\gamma/m \in \mathcal D$.  The number of such hemispheres is in the order of $\vert d_K\vert^2$.  However, most of them will be covered by the others.

\item[(II.2)] For each $S_\lambda(\gamma/m, \infty)$ drawn in (II.1), use Sage to draw the additional hemispheres $S_\lambda(\gamma/m + \mu, \infty)$ for all $\mu \in \{\pm\pi_1, \pm\pi_2, \pm(\pi_1 + \pi_2), \pm(\pi_1 - \pi_2) \}$.  As in (II.1), most of these hemispheres are covered by the others.

\item[(II.3)]  For the sake of discussion, we call the hemispheres from (II.1) and (II.2) {\em black} and {\em white}, respectively.  The result of (II.1) and (II.2), when seen directly from the top, will be a picture showing a set $J_\lambda$, which is a union of subsets $A_\lambda(\gamma/m)$ of some black hemispheres $S_\lambda(\gamma/m, \infty)$,  surrounded by some white hemispheres.   Figure~\ref{fig D=39 H(infty)} shows the resulting picture for the case $K = \Q(\sqrt{-39})$ and $\lambda = \infty$.

\item[(II.4)] Show that $J_\lambda$ is indeed $I_\lambda$.  This is done by showing that every hemisphere $S_\lambda(\delta/k, \infty)$ does not cover strictly any point on $J_\lambda$, and the orthogonal projection of $J_\lambda$ on $\C$ is a fundamental domain of $L_\lambda^{-1}\Gamma(\lambda)L_\lambda$.  The later is quite straightforward, since the main difficulty involves only the determination of the intersections of hemispheres.   For the former,  let $\zeta$ be the height of the lowest point of $J_\lambda$.  The radius of $S_\lambda(\delta/k, \infty)$ is smaller than or equal than $\sqrt{1/k}$ by Lemma~\ref{lem 4.3 Vogtmann}.  Therefore, we only need to check those $S_\lambda(\delta/k, \infty)$  such that $k$ is bounded above by $1/\zeta^2$ and $\delta/k$ is close enough to some $A_\lambda(\gamma/m)$.  There are only finitely many these hemispheres, and hence the checking can be done in a finite number of steps.
\end{enumerate}

The above indicates that the complexity of Step (II) increases when $|d_K|$ increases since more hemispheres will be needed.  However, sometimes the computation can be simplified by exploiting the symmetries explained in Lemma~\ref{lem H_lambda and H_bar-lambda} and Theorem~\ref{thm class number two Vogtmann}.  As an illustration of our method, we determine $\gamma_K^p$ and the corresponding projective extreme forms when $K=\Q(\sqrt{-39})$, whose class number is four.

Fix $K = \Q(\sqrt{-39})$ from now on.  We follow the notation introduced in the previous sections. In particular, for $\delta,\beta\in\mathcal O_K$, $[\delta/\beta]$ denotes the class of the ideal $\langle\delta,\beta\rangle$.  The class group of $K$ is cyclic of order $4$ generated by $\theta=[\omega/2]$ with $\omega=(1+\sqrt{-39})/2$.   More precisely,
$$
\cl_K=\left\{\theta^0=[\infty], \;\theta^1=\left[\frac\omega2\right], \;\theta^2=\left[\frac{1+\omega}3\right], \;\theta^3=\left[\frac{1+\omega}2\right]\right\}.
$$

\smallskip
\noindent \fbox{Case $\lambda_0:=\infty$}
As is explained in Remark \ref{Iinfinity}, we may choose $\mathcal D$ to be $\{x+y\omega\in\C\;:\;0\leq x,y<1\}$.
We need to consider $H(\infty)\cap H(\gamma/m)$ for $0<m<39/2=19.5$ and $\gamma \in \mathcal O_K$ such that $\gamma/m \in \mathcal D$.

\setlength{\unitlength}{1cm}
\begin{figure}[!htb]
\caption{$I(\infty)$ for $K=\Q(\sqrt{-39})$}
\label{fig D=39 H(infty)}
\rule{0.5cm}{0cm}\\
\includegraphics{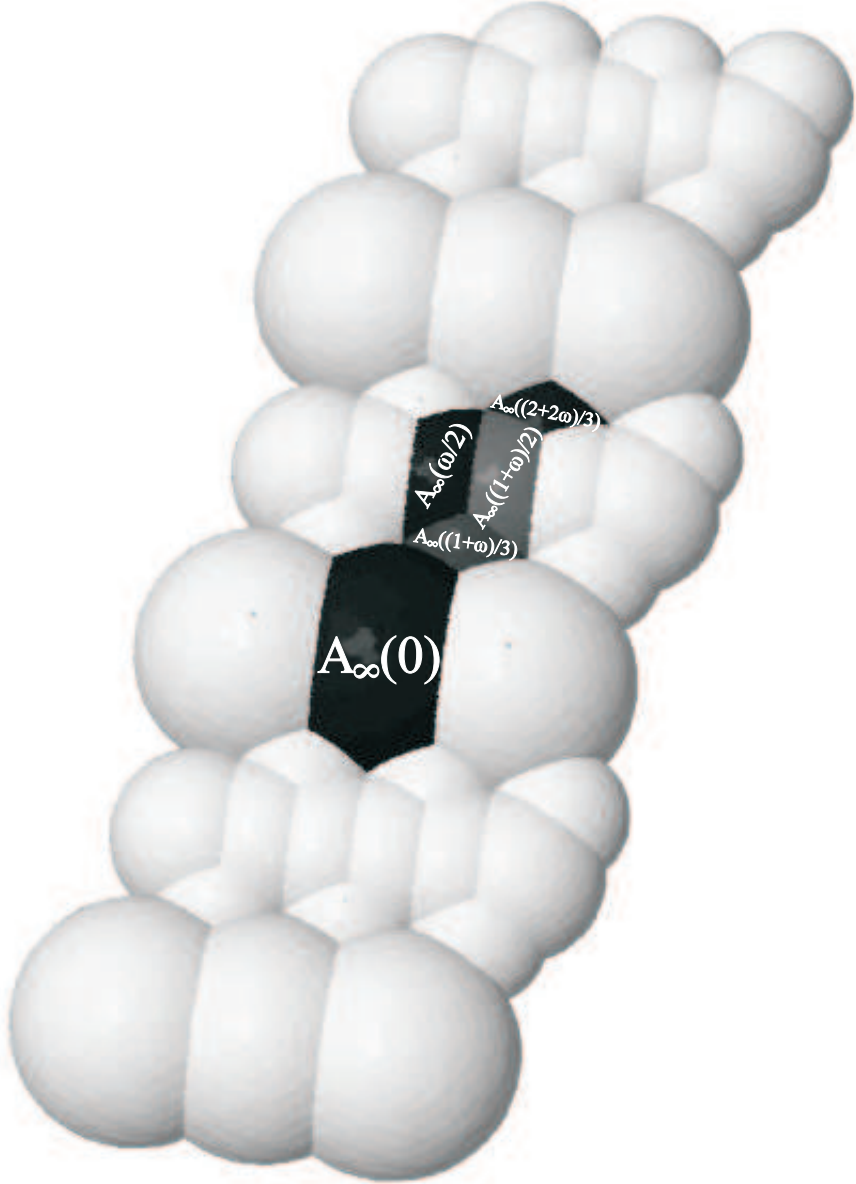}
\end{figure}

Figure~\ref{fig D=39 H(infty)} shows all the black and white hemispheres from (II.1) and (II.2) that are visible directly from the top, and the resulting subset $J_\infty$ from (II.3) which is the union of five $A_\infty(\gamma/m)$, where $\gamma/m$ are $0, \frac{1+\omega}3, \frac{\omega}2, \frac{1+\omega}2$, and $\frac{2+2\omega}3$.  Note that for each of these five $\gamma/m$, $A_\infty(\gamma/m)$ is the subset of $S(\infty, \gamma/m)$ which contains points that are outside all the other black and white hemispheres from (II.1) and (II.2).  The orthogonal projection of $J_\infty$ onto $\mathbb C$ is shown in Figure~\ref{fig D=39 H(infty)inC}.  It is not hard to check that $J_\infty$ is a fundamental domain for the action of $\Gamma(\infty)$ on $\C$.

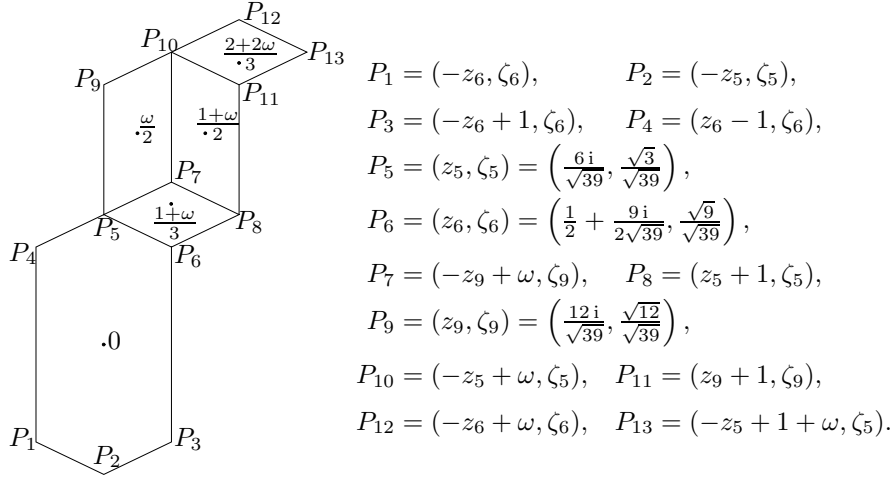
\begin{figure}[!htb]
\caption{Projection of $I(\infty)$ onto $\C$ for $K=\Q(\sqrt{-39})$}
\label{fig D=39 H(infty)inC}
\rule{0.5cm}{0cm}\\
\begin{minipage}{0.38\textwidth}
\setlength{\unitlength}{1.8cm}
\rule{3mm}{0pt}
\begin{picture}(2.5,3.5)
  \path(1.5,1.72)(1.5,0.28)(1,0.04)(0.5,0.28)(0.5,1.72) (1,1.96)(1.5,1.72)(2,1.96)(2,2.92)(1.5,3.16)(1,2.92) (1,1.96)(1.5,2.20)(2,1.96)
  \path(1.5,2.20)(1.5,3.16)(2,3.40)(2.5,3.16)(2,2.92)
  \put(0.30,0.25){$P_1$}  \put(0.30,1.62){$P_4$}  \put(1.52,0.25){$P_3$}  \put(1.52,1.60){$P_6$}  \put(0.92,0.12){$P_2$}  \put(0.92,1.80){$P_5$}  \put(1.97,1.85){$P_8$}  \put(1.52,2.22){$P_7$}  \put(0.78,2.90){$P_9$}  \put(2.02,2.80){$P_{11}$}  \put(1.27,3.20){$P_{10}$}  \put(2.50,3.10){$P_{13}$}  \put(2.03,3.40){$P_{12}$}
  \put(1.00,1.00){\circle*{0.02}} \put(1.03,0.97){$0$}
  \put(1.50,2.04){\circle*{0.02}} \put(1.35,1.85){$\tfrac{1+\omega}{3}$}
  \put(2.00,3.08){\circle*{0.02}} \put(1.87,3.10){$\tfrac{2+2\omega}{3}$}
  \put(1.25,2.56){\circle*{0.02}} \put(1.25,2.56){$\tfrac{\omega}{2}$}
  \put(1.75,2.56){\circle*{0.02}} \put(1.67,2.58){$\tfrac{1+\omega}{2}$}
\end{picture}
\end{minipage}
\begin{minipage}[b]{0.6\textwidth}
\setlength{\unitlength}{1cm}
\newlength{\altura}
\setlength{\altura}{14pt}
$$
\begin{array}{r@{\;=\;}lr@{\;=\;}l}
  \rule{0pt}{\altura}
  P_1&(-z_6,\zeta_6),&
  P_2&(-z_5,\zeta_5),\\
  \rule{0pt}{\altura}
  P_3&(-z_6+1,\zeta_6),&
  P_4&(z_6-1,\zeta_6),\\
  \rule{0pt}{\altura}
  P_5&\multicolumn{3}{l}{\!\!\!(z_5,\zeta_5)= \left(\tfrac{6\mi}{\sqrt{39}},\tfrac{\sqrt3}{\sqrt{39}}\right),}\\
  \rule{0pt}{\altura}
  P_6&\multicolumn{3}{l}{\!\!\!(z_6,\zeta_6)= \left(\tfrac12+\tfrac{9\mi}{2\sqrt{39}},\tfrac{\sqrt9}{\sqrt{39}}\right),}\\
  \rule{0pt}{\altura}
  P_7&(-z_9+\omega,\zeta_9),&
  P_8&(z_5+1,\zeta_5),\\
  \rule{0pt}{\altura}
  P_9&\multicolumn{3}{l}{\!\!\!(z_9,\zeta_9)= \left(\tfrac{12\mi}{\sqrt{39}},\tfrac{\sqrt{12}}{\sqrt{39}}\right),}\\
  \rule{0pt}{\altura}
  P_{10}&(-z_5+\omega,\zeta_5),&
  P_{11}&(z_9+1,\zeta_9),\\
  \rule{0pt}{\altura}
  P_{12}&(-z_6+\omega,\zeta_6),&
  P_{13}&(-z_5+1+\omega,\zeta_5).
\end{array}
$$
\end{minipage}
\end{figure}

We checked that (see Figure~\ref{fig D=39 H(infty)inC}) the height of the lowest point of $J_\infty$ is $\frac{1}{\sqrt{13}}$.
Now, the radius of $S(\delta/k,\infty)$ is $\sqrt{\norm{\delta,k}/k^2}$, which is $\le 1/\sqrt{k}$ by Lemma~\ref{lem 4.3 Vogtmann}. Thus we only need to check finitely many $S(\delta/k, \infty)$ for which $k \leq 13$ and $\vert \delta/k - z\vert < 1/\sqrt{13}$ for some $(z,\zeta) \in J_\infty$.  After carrying out these computations, which we do not include here for brevity, we conclude that $J_\infty$ is indeed the $I(\infty)$ given by Theorem~\ref{thm3:Vogtmann}.

\smallskip

\noindent\fbox{Case $\lambda_1:=\omega/2$} For this cusp,
$$
L_{\lambda_1}=\left(\begin{matrix} \omega&1\\2&0 \end{matrix}\right)
\qquad\text{and}\qquad
2{\langle\omega, 2 \rangle^{-2}}=2\,\Z\oplus\tfrac{1-\omega}2\,\Z.
$$
In this case, $\mathcal D=\{x+y({1-\omega})/{2}\in\C\;:\;0\leq x,y<1\}$ is a fundamental domain of $2{\langle\omega,2\rangle^{-2}}$ in $\C$.

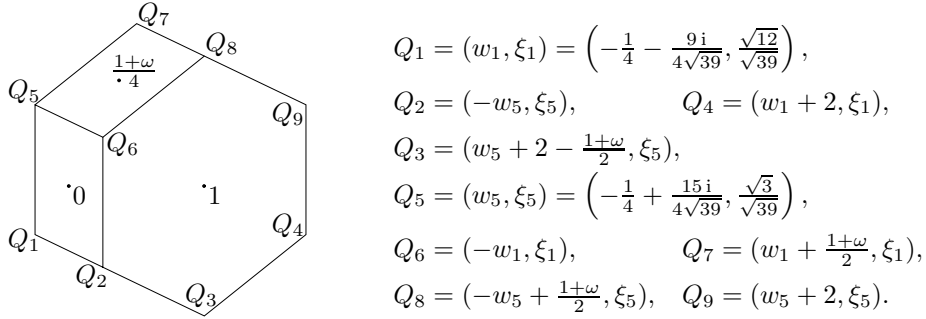
\begin{figure}[!htb]
\caption{Projection of $I_{\lambda_1}$ onto $\C$ for $K=\Q(\sqrt{-39})$}
\label{fig D=39 H(omega/2)inC}
\rule{0.5cm}{0cm}\\
\begin{minipage}{0.38\textwidth}
\setlength{\unitlength}{1.8cm}
\begin{picture}(2.5,2.3)
  \path(0.75,0.40)(1.50,0.04)(2.25,0.64)(2.25,1.60)(1.00,2.20)(0.25,1.60)(0.25,0.64)(0.75,0.40)(0.75,1.36)(0.25,1.60)
  \path(0.75,1.36)(1.50,1.96)
  \put(0.50,1.00){\circle*{0.02}} \put(0.53,0.87){$0$}
  \put(1.50,1.00){\circle*{0.02}} \put(1.53,0.87){$1$}
  \put(0.87,1.78){\circle*{0.02}} \put(0.81,1.80){$\frac{1+\omega}{4}$}
  \put(0.05,0.55){$Q_1$}
  \put(0.55,0.28){$Q_2$}
  \put(1.36,0.14){$Q_3$}
  \put(2.00,0.66){$Q_4$}
  \put(0.05,1.65){$Q_5$}
  \put(0.78,1.26){$Q_6$}
  \put(1.01,2.22){$Q_7$}
  \put(1.51,2.00){$Q_8$}
  \put(2.00,1.48){$Q_9$}
\end{picture}
\end{minipage}
\setlength{\unitlength}{1cm}
\begin{minipage}[b]{0.57\textwidth}
\setlength{\altura}{14pt}
$$
\begin{array}{r@{\;=\;}lr@{\;=\;}l}
  \rule{0pt}{\altura}
  Q_1&\multicolumn{3}{l}{\!\!\!(w_1,\xi_1)= \left(-\tfrac14-\tfrac{9\mi}{4\sqrt{39}},\tfrac{\sqrt{12}}{\sqrt{39}}\right),}\\
  \rule{0pt}{\altura}
  Q_2&(-w_5,\xi_5),&
  Q_4&(w_1+2,\xi_1),\\
  \rule{0pt}{\altura}
  Q_3&\multicolumn{3}{l}{\!\!\!(w_5+2-\tfrac{1+\omega}2,\xi_5),}\\
  \rule{0pt}{\altura}
  Q_5&\multicolumn{3}{l}{\!\!\!(w_5,\xi_5)= \left(-\tfrac14+\tfrac{15\mi}{4\sqrt{39}},\tfrac{\sqrt3}{\sqrt{39}}\right),}\\
  \rule{0pt}{\altura}
  Q_6&(-w_1,\xi_1),&
  Q_7&(w_1+\tfrac{1+\omega}2,\xi_1),\\
  \rule{0pt}{\altura}
  Q_8&(-w_5+\tfrac{1+\omega}2,\xi_5),&
  Q_9&(w_5+2,\xi_5).
\end{array}
$$
\end{minipage}
\end{figure}

From (II.1), (II.2), and (II.3), we obtain $J_{\lambda_1}$ as the union of  $A_{\lambda_1}(0)$, $A_{\lambda_1}(\omega/2)$, and $A_{\lambda_1}(1)$.  Here for $\mu \in \{0, \omega/2, 1\}$, $A_{\lambda_1}(\mu)$ is the subset of $S_{\lambda_1}(\infty, \mu)$ which contains points that are outside all the other hemispheres from (II.1) and (II.2).  The orthogonal projection of $J_{\lambda_1}$ to $\C$ is shown in Figure~\ref{fig D=39 H(omega/2)inC}, which can be seen to be a fundamental domain for the action of $L_{\lambda_1}^{-1}\Gamma(\lambda_1)L_{\lambda_1}$ on $\C$.  With an argument similar to the one used in the previous case, we obtain that
$$
I_{\lambda_1}=A_{\lambda_1}(0)\cup A_{\lambda_1}(\tfrac\omega2)\cup A_{\lambda_1}(1)
\qquad\text{and}\qquad
I(\lambda_1) = L_{\lambda_1}(I_{\lambda_1}).
$$

\smallskip

\noindent \fbox{Case ${\lambda_2}:=(1+\omega)/3$}
Note that $[{\lambda_2}]^2=1$.
Following the proof of Theorem~\ref{thm class number two Vogtmann}, we take
$$
g_2=\begin{pmatrix}
  1+\omega&2-\omega\\ 3&-1-\omega
\end{pmatrix}.
$$
This matrix satisfies the requirements given in \eqref{eq g class number two} since $\langle1+\omega,3\rangle=\langle2-\omega,-1-\omega\rangle$ and $\det(g)=3=N_K (\langle 1 + \omega ,3\rangle)$. Furthermore, $g(\infty)=\lambda_2$ and $g\left(\lambda_2\right)=\infty$.
Finally, Theorem \ref{thm class number two Vogtmann} implies that
$H(\lambda_2)=g_2(H(\infty))$ and $g_2(I(\infty))$ is a fundamental domain for the action of the group $\Gamma(\lambda_2)$ on $\partial H(\lambda_2)$.

\smallskip

\noindent\fbox{Case ${\lambda_3}:=(1+\omega)/2$}
Lemma \ref{lem H_lambda and H_bar-lambda} now yields that $\partial H_{\lambda_3}(\infty)$ is the reflection of $\partial H_{\lambda_1}(\infty)$ through the plane $\mathrm{Im}(z)=0$, since $[\overline{{\lambda}}_3]= \theta^{-3}= \theta= [{\lambda_1}]$.
Thus $L_{\lambda_3} \left(\overline{I_{\lambda_1}}\right) = L_{\lambda_3} \left(\overline{L_{\lambda_1}^{-1}(I(\lambda_1))}\right)$ is a fundamental domain for $\Gamma(\lambda_3)$ on $\partial H(\lambda_3)$.

\bigskip

Summarizing, we have determined the fundamental domains (given by Theorem~\ref{thm3:Vogtmann})
$$
I(\infty)
\qquad\text{and}\qquad
I(\lambda_1)=L_{\lambda_1}(I_{\lambda_1})
=\left(\begin{smallmatrix}\omega&1\\2&0\end{smallmatrix}\right)\cdot I_{\lambda_1}
$$
for $\partial H(\infty)/\Gamma(\infty)$ and $\partial H(\lambda_1)/\Gamma(\lambda_1)$ respectively, and also we have proved that
$$
g_2(I(\infty))
\qquad\text{and}\qquad
L_{\lambda_3}\left(\overline{I_{\lambda_1}}\right)
=\left(\begin{smallmatrix}1+\omega&1\\2&0\end{smallmatrix}\right) \cdot\overline{I_{\lambda_1}}
$$
are fundamental domains for $\partial H(\lambda_2)/\Gamma(\lambda_2)$ and $\partial H(\lambda_3)/\Gamma(\lambda_3)$ respectively.
Let $I$ be the union of these four sets.

We cannot assert that $I$ is a fundamental domain for $\Gamma=\mathrm{SL}_2(\mathcal O_K)$ on the minimal incidence set $\MM$ (see \eqref{eq minimal incidence set}), but it is evident that every point in $\MM$ is $\mathrm{SL}_2(\mathcal O_K)$-equivalent to one point in $I$.
Our next goal is to determine a set of representatives of the set of vertices of $I$ under the action of $\Gamma$.
Recall that the vertices of $I$ are the projective extreme points, which are in a one to one correspondence with the projective extreme forms by Proposition~\ref{prop2:extreme-points}.

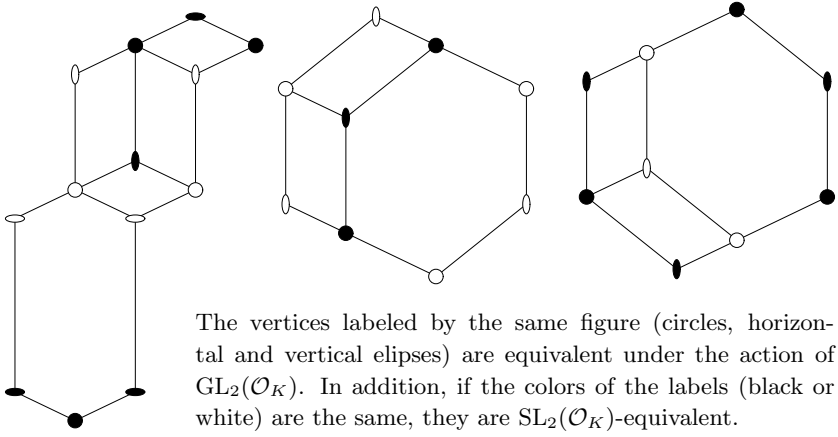
\begin{figure}
\caption{Proj. extreme points under $\mathrm{SL}_2(\mathcal O_K)$-equivalence, $K=\Q(\sqrt{-39})$}
\label{fig extreme points/SL_2}
\rule{0.5cm}{0cm}\\
\newcommand{\circO}{\circle*{0.12}}
\newcommand{\circV}{\circle*{0.12}}
\newcommand{\triaO}{\ellipse*{0.06}{0.16}}
\newcommand{\triaV}{\ellipse*{0.06}{0.16}}
\newcommand{\cuadO}{\ellipse*{0.16}{0.06}}
\newcommand{\cuadV}{\ellipse*{0.16}{0.06}}
\setlength{\unitlength}{1.6cm}
\begin{picture}(7.5,3.5)
\path (1.50,1.72) (1.50,0.28) (1.00,0.04) (0.50,0.28) (0.50,1.72) (1.00,1.96) (1.50,1.72)
      (2.00,1.96) (2.00,2.92) (1.50,3.16) (1.00,2.92) (1.00,1.96) (1.50,2.20) (2.00,1.96)
\path (1.50,2.20) (1.50,3.16) (2.00,3.40) (2.50,3.16) (2.00,2.92)
\path (3.25,1.60) (4.00,1.24) (4.75,1.84) (4.75,2.80) (3.50,3.40)
      (2.75,2.80) (2.75,1.84) (3.25,1.60) (3.25,2.56) (2.75,2.80)
\path (3.25,2.56) (4.00,3.16)
\path (5.75,3.10) (6.50,3.46) (7.25,2.86) (7.25,1.90) (6.00,1.30)
      (5.25,1.90) (5.25,2.86) (5.75,3.10) (5.75,2.14) (5.25,1.90)
\path (5.75,2.14) (6.50,1.54)
\filltype{black}
\put(1.50,3.16){\circO} \put(2.50,3.16){\circO} \put(1.00,0.04){\circO}
\put(0.50,0.28){\cuadO} \put(1.50,0.28){\cuadO} \put(2.00,3.40){\cuadO}
\put(1.50,2.20){\triaO}
\put(3.25,1.60){\circO} \put(4.00,3.16){\circO}
\put(3.25,2.56){\triaO}
\put(5.25,1.90){\circO} \put(7.25,1.90){\circO} \put(6.50,3.46){\circO}
\put(7.25,2.86){\triaO} \put(5.25,2.86){\triaO} \put(6.00,1.30){\triaO}
\filltype{white}
\put(1.00,1.96){\circV} \put(2.00,1.96){\circV}
\put(0.50,1.72){\cuadV} \put(1.50,1.72){\cuadV}
\put(1.00,2.92){\triaV} \put(2.00,2.92){\triaV}
\put(4.00,1.24){\circV} \put(2.75,2.80){\circV} \put(4.75,2.80){\circV}
\put(4.75,1.84){\triaV} \put(2.75,1.84){\triaV} \put(3.50,3.40){\triaV}
\put(5.75,3.10){\circV} \put(6.50,1.54){\circV}
\put(5.75,2.14){\triaV}
\put(2,0){\parbox[b]{8.5cm}{{\small The vertices labeled by the same figure (circles, horizontal and vertical elipses) are equivalent under the action of $\mathrm{GL}_2(\mathcal O_K)$. In addition, if the colors of the labels (black or white) are the same, they are $\mathrm{SL}_2(\mathcal O_K)$-equivalent.}}}
\end{picture}
\end{figure}

It is easy to see that $\{P_1, P_2, P_5, P_6, P_7,P_9\}$ (see Fig. 2) is a set of representatives of the set of vertices of $I(\infty)$ under the action of $\Gamma(\infty)$. Note that the group $\mathrm{GL}_2(\mathcal O_K)$ is generated by the elements of $\Gamma$ and the matrix $\left(\begin{smallmatrix}-1&0\\0&1\end{smallmatrix}\right)$ which acts on $\mathrm H^3$ as $(z,\zeta)\mapsto(-z,\zeta)$.
This implies that $\{P_2, P_5\}$, $\{P_1, P_6\}$, and $\{P_7, P_9\}$ are $\mathrm{GL}_2(\mathcal O_K)$-orbits, and  therefore $\{P_5, P_6, P_9\}$ is a set of representatives of the set of vertices of $I(\infty)$ under the action of $\mathrm{GL}_2(\mathcal O_K)$. These facts are illustrated in the picture on the left in \mbox{Figure \ref{fig extreme points/SL_2}}.

The vertices of $g_2(I(\infty))$ are of the form $\{g_2(P_j)\,:\, 1\leq j\leq 13\}$.
It can be checked that  $A_\infty(\frac{1+\omega}{3})$ on $S((1+\omega)/3,\infty)$ is mapped into itself by $g_2$.
More precisely, $g_2(P_5)=P_8$, $g_2(P_8)=P_5$, $g_2(P_6)=P_6$ and $g_2(P_7)=P_7$.
On the other hand, it is easy to check that $g_2hg_2^{-1}\in\Gamma$ for all $h\in\Gamma$.
This implies that two points are $\Gamma$-equivalent if and only if their images under $g_2$ are $\Gamma$-equivalent. Therefore any vertex in $g_2(I(\infty))$ is $\Gamma$-equivalent to some vertex in $I(\infty)$.

It is easily seen that (or, see \cite[Page 406]{Vogtmann})
$$
L_{\lambda_1}^{-1}\,\Gamma(\lambda_1)\,L_{\lambda_1}= \left\{\begin{pmatrix} -2 & -4x\\ 0& -2 \end{pmatrix}\,:\, x \in \langle \omega, 2 \rangle^{-2}\right\}.
$$
We now describe the action of $L_{\lambda_1}^{-1}\,\Gamma(\lambda_1)\, L_{\lambda_1}$ on the vertices of $I_{\lambda_1}$ in \mbox{Figure \ref{fig extreme points/SL_2}.}
It is a simple matter to check that any vertex in $I_{\lambda_1}$ is equivalent under $L_{\lambda_1}^{-1}\,\Gamma(\lambda_1)\,L_{\lambda_1}$ to one of $\{Q_1,Q_2,Q_5,Q_6\}$ (see Fig. 3). One can compute that $L_{\lambda_1}(Q_1)=P_9$, $L_{\lambda_1}(Q_2)=P_{10}$, $L_{\lambda_1}(Q_5)=P_5$ and $L_{\lambda_1}(Q_6)=P_7$.  This tells us that each of the vertices of $L_{\lambda_1}(I_{\lambda_1})$ is equivalent under $\Gamma$ to a point in $I(\infty)$ and therefore to a point in $\{P_1, P_2, P_5, P_6, P_7, P_9\}$. The same proof works for $\lambda_3$.

In conclusion, the points $P_1, P_2, P_5, P_6, P_7, P_9$ are the projective extreme points up to $\Gamma=\mathrm{SL}_2(\mathcal O_K)$-equivalence.
These points contain all the necessary information to obtain the projective Hermite constant $\gamma_K^p$ and the projective extreme forms.
By \eqref{eq2:1st-red-Mendoza}, $\gamma_K^p = \max_{P\in I} \eta_K(P)$, where $\eta_K(P)$ is given by \eqref{eq eta}.
On the other hand, this maximum is attained at a vertex of $I$.
All the vertices have representatives in $I(\infty)$, and the distance is $\Gamma$-invariant.
Therefore,
\begin{equation*}
  \gamma_K^p = \max_{j\in\{1,2,5,6,7,9\}} \dist(P_j,\infty)=\max_{j\in\{1,2,5,6,7,9\}} \frac{1}{\zeta_j}=\sqrt{\frac{39}3}=\sqrt{13}.
\end{equation*}
The maximum is attained at the vertices $P_2$ and $P_5$.  By applying the map $\Psi=\Psi_1$ in Theorem~\ref{Thm bijection H3 and hermitian forms}, we obtain the two associated absolutely projective extreme forms of determinant one:
$$\Psi(P_2) = \begin{pmatrix} \sqrt{13}&-\sqrt{12}\mi\\ -\sqrt{12}\mi&\sqrt{13} \end{pmatrix},\quad
\Psi(P_5) =  \begin{pmatrix} \sqrt{13}& \sqrt{12}\mi\\ \sqrt{12}\mi&\sqrt{13} \end{pmatrix}.$$
Both of them attain their projective minima at $\left(\begin{smallmatrix}1\\0 \end{smallmatrix}\right)$, which is the vector associated with the cusp $\infty=1/0$.


\section{Comparison with other Hermite constants}\label{sec5:comparison}

In this section we compare the constants $\gamma_K$ and $\gamma_K^p$ given in \eqref{eq1:gamma_K} and \eqref{eq1:gamma_K^p} with other generalizations of the Hermite constant for any number field $K$.  From now on, $K$ is an arbitrary number field of degree $m=r+2s$ and $\mathcal O_K$ is the ring of integers in $K$.  The discriminant and the class number of $K$ are denoted by $d_K$ and $h_K$ respectively.
Let $\{\sigma_1, \ldots, \sigma_m\}$ be the set of embeddings of $K$ into the complex numbers.
Our convention is that $\sigma_1, \ldots, \sigma_r$ are real and $\sigma_{r+1}, \ldots, \sigma_{r+2s}$ are complex with $\sigma_{r+j} = \overline{\sigma_{r+s+j}}$ for $j = 1, \ldots, s$.

An $(r+s)$-tuple $S=(S_1,\dots,S_{r+s})$, where $S_1,\dots, S_r$ are $n\times n$ positive definite real symmetric  matrices and $S_{r+1},\dots,S_{r+s}$ are $n\times n$ positive definite hermitian matrices, is called a \emph{Humbert form} over $K$ of dimension $n$.
The set of all such forms is denoted by $\mathcal P_{K,n}$.  Let $S$ be an $n$-dimensional Humbert form over $K$.
For $x\in\mathcal O_K^n$,
$$
S[x]:=\prod_{i=1}^r S_i[x^{\sigma_i}]\cdot \left(\prod_{i=r+1}^{r+s} S_i[x^{\sigma_i}]\right)^2,
$$
where $x^{\sigma_i}$ is the column vector with coordinates $x_1^{\sigma_i},\dots,x_n^{\sigma_i}$ and $S_i[x]=x^*S_ix$ with $x^*$ being the conjugate transpose of $x$.   The determinant of $S$ is defined as
$$\det(S)=\prod_{i=1}^r \det S_i\cdot \left(\prod_{i=r+1}^{r+s} \det S_i\right)^2.$$

In \cite{I}, Icaza defines the $n$-dimensional \emph{Hermite-Humbert constant} of $K$ by
\begin{equation}\label{eq5:cte-Icaza}
\gamma(n,K) = \sup_{S\in\mathcal P_{K,n}} \min_{x\in\mathcal O_K^n\smallsetminus\{0\}} \frac{S[x]}{\det(S)^{1/n}}.
\end{equation}
Humbert reduction theory for positive definite quadratic forms over number fields implies that $\gamma(n,K)<\infty$.
When $K=\Q$, the Hermite-Humbert constant coincides with the classical Hermite constant, which is known for $2\leq n\leq 8$ and for $n=24$. When $n=2$ and $K$ is a totally real quadratic number field, $\gamma(2,K)$ is known in several cases (see \cite{BCIO}, \cite{CIO}, \cite{PohstWagner1}, and \cite{PohstWagner2}).

If $\mathfrak b$ is a fractional ideal of $K$, let $N_K(\mathfrak b)$ be its ideal norm.  The fractional ideal generated by the coordinates of a nonzero vector $x \in K^n$ is denoted by $\langle x \rangle$.  We define the
$n$-dimensional \emph{projective Hermite-Humbert constant} of $K$ to be
\begin{equation}\label{eq5:cte-proj}
\gamma^p(n,K) = \sup_{S\in\mathcal P_{K,n}} \min_{x\in\mathcal O_K^n\smallsetminus\{0\}} \frac{S[x]}{\norm{x}^2\det(S)^{1/n}}.
\end{equation}

It is clear that $\gamma^p(n,K) \leq \gamma(n,K)$ and they coincide when the class number is one.
Furthermore, it follows that when $K$ is an imaginary quadratic field, we have
$$
\gamma(2,K)^{\frac12} = \gamma_K
\quad\text{and}\quad
\gamma^p(2,K)^{\frac12} = \gamma^p_{K}
$$
Thus, $\gamma^p(n,K)^{\frac12}$ can be viewed as the generalization of $\gamma^p_K$ in the context of Humbert forms.

On the other hand, using the twisted height on the Grassmanian manifold, Thunder~\cite{Thunder} defines constants that are analogs of the Hermite constants.   For any place $v$ of $K$, let $K_v$ be the completion of $K$ with respect to $v$, and let $\vert \,\, \vert_v$ be the absolute value on $K_v$ normalized so that $\mu(\alpha C) = \vert \alpha \vert_v\mu(C)$ where $\mu$ is a Haar measure on $K_v$ and $C$ is a compact subset of $K_v$ with nonzero Haar measure.
For any $x\in K_v^n$, the local height $H_v(x)$ is defined by
$$H_v(x) = \left \{
\begin{array}{ll}
\displaystyle{\left(\sum_{i=1}^n \vert x_i\vert_v^2\right)^{1/(2m)}} & \mbox{ if $v$ is real };\\
\displaystyle{\left(\sum_{i = 1}^n \vert x_i \vert_v\right)^{1/m}} & \mbox{ if $v$ is complex };\\
\displaystyle{\left(\max_{1\leq i \leq n} \vert x_i\vert_v\right)^{1/m}} & \mbox{ if $v$ is finite}.
\end{array}
\right.$$
Let $\A$ be the ring of adeles of $K$, and $|\cdot|_\A$ be the adelic norm given by the product over all places.  For any $g \in \textrm{GL}_n(\A)$, the global twisted height $H_g$ is defined by $H_g(x) = \prod_v H_v(gx)$ for any $x \in K^n$.    Now, as is described in \cite{Ohno-Watanabe}, there is a double coset decomposition
$$\text{GL}_n(\A) = \bigcup_{i=1}^{h_K} \textrm{GL}_n(\A_\infty)\lambda_i\textrm{GL}_n(K)$$
with $\lambda_1 \in \text{GL}_n(\A_\infty) = \text{GL}_n(K_\infty)\prod_{v\nmid \infty} \text{GL}_n(\mathcal O_{K_v})$.
Under this notation, Thunder's  constant is defined by
\begin{equation}\label{eq5:cte-Thunder}
\gamma^T(n,K):= \max_{g\in \textrm{GL}_n(K_\infty)} \min_{x\in K^n\smallsetminus\{0\}} \frac{H_{g\lambda_i}(x)^2}{\vert \det g\lambda_i \vert_\A^{2/(nm)}}.
\end{equation}

Let $S$ be an $n$-dimensional Humbert form over $K$.  We may interpret $S$ as a sequence $\{S_v : v \mid \infty\}$, with $S_v$ a positive definite $n\times n$ symmetric or Hermitian matrix, depending on whether  $v$ is real or complex.  Then each $S_v$ can be written as $g_v^*g_v$ for some $g_v \in \text{GL}_n(K_v)$, and $N_K(\langle x \rangle) = \prod_{v\nmid \infty} H_v(x)^{-1}$ for every $x \in K^n\smallsetminus\{0\}$.
Therefore,
\begin{equation}\label{eq5:gamma^p-height}
\gamma^p(n,K) = \max_{g \in \text{GL}_n(K_\infty)} \min_{x \in K^n\smallsetminus\{0\}} \frac{H_g(x)^{2m}}{\vert \det g\vert_\A^{2/n}},
\end{equation}
where $\textrm{GL}_n(K_\infty) = \prod_{v \mid \infty} \textrm{GL}_n(K_v)$ regarded as a subgroup of $\textrm{GL}_n(\A)$.
As a result of comparing \eqref{eq5:cte-Thunder} and \eqref{eq5:gamma^p-height}, we have the inequality
$$\gamma^p(n,K) \leq \gamma^T(n,K)^m,$$
and the equality holds if the field $K$ has class number one.

In \cite{W-Lie, W-Fund},  Watanabe introduces a constant $\gamma^W(G,\pi)$ attached to a connected linear algebraic group $G$ defined over $K$ and an absolutely irreducible strongly $K$-rational representation $\pi$ of $G$,  which can be viewed as an analog of the classical Hermite constants.  Indeed, when $G = \textrm{GL}_n$ and $\pi = \rho$ is the standard representation of $\textrm{GL}_n$, then $\gamma^W(\textrm{GL}_n, \rho)$ is exactly Thunder's constant $\gamma^T(n,K)$.
If we take a maximal $K$-parabolic subgroup $Q$ of $G$ and a representation $\pi: G \rightarrow \textrm{GL}(V_\pi)$ such that the stabilizer of the highest weight line of $\pi$ is $Q$, then
\begin{equation}\label{eq5:cte-Watanabe}
\gamma^W(G,\pi) = \max_{g \in G(\A)^1} \min_{x \in Q(K)\backslash G(K)} H_\pi(xg)^2
\end{equation}
where $H_\pi$ is the global height function on $V_\pi$.

Let $g \in \textrm{GL}_n(K_\infty)$ be given.
For each $v \mid \infty$, there exists $a_v \in K_v$ such that $\det(a_vg_v) = 1$.
In particular, the matrix $h_v := a_vg_v$ is in $\textrm{SL}_n(K_v)$.
Define $h \in \textrm{SL}_n(K_\infty) \leq \textrm{SL}_n(\A)$ by setting $h_v = a_vg_v$ for all $v \mid \infty$ and, of course, $h_v = 1$ for all $v \nmid \infty$.
Then for each $v \mid \infty$,
$$
H_{h_v}(x)^2 = H_v(a_vg_vx)^{2/m} = \vert a_v\vert_v^{2/m} H_v(g_v x)^{2/m} = \frac{H_{g_v}(x)^2}{\vert \det g_v \vert_v^{2/(nm)}}.
$$
Since $\textrm{SL}_n$ has the strong approximation property (assuming $n \geq 2$), therefore $\textrm{SL}_n(\A) = \textrm{SL}_n(\A_\infty)\textrm{SL}_n(K)$.
As a result, we can rewrite \eqref{eq5:gamma^p-height} as
\begin{equation}\label{eq5:proj-Watanabe}
\gamma^p(n,K) = \max_{h \in \textrm{SL}_n(\A)} \min_{x\in K^n\smallsetminus\{0\}} H_h(x)^{2m}.
\end{equation}

If we specialize \eqref{eq5:cte-Watanabe} to the group $G = \textrm{SL}_n$ and $\rho$ the natural representation of $\textrm{SL}_n$, then, in view of \eqref{eq5:proj-Watanabe} and the observation that $\textrm{SL}_n(\A) = \textrm{SL}_n(\A)^1$ because $\textrm{SL}_n$ is semisimple, we have
$$
\gamma^p(n,K) = \gamma^W({\textrm{SL}_n},\rho)^m.
$$
We therefore obtain the exact values of $\gamma^W(\textrm{SL}_2,\rho)$ for imaginary quadratic fields covered by Theorem~\ref{thm1:ctes}.

\section*{Acknowledgements}
The research of the second author was partially supported by FONDECYT 1080015 and CONICYT ACT 56.  This author also extends her thanks to the Department of Mathematics and Computer Science of Wesleyan University for their hospitality during her visits funded by the Van Vleck Research Fund.  The third author would like to thank Universidad de Talca in Chile for their generous hospitality during three visits in 2009.  The authors thank Roberto Miatello and Takao Watanabe for their helpful comments and discussion.

\end{document}